\documentclass[11pt]{article}

\usepackage[utf8]{inputenc}
\usepackage[T1]{fontenc}
\usepackage{mathpazo}

\RequirePackage[english]{babel}

\RequirePackage[a4paper,margin=2.5cm]{geometry}
\RequirePackage{parskip}

\newcommand{\affiliation}{\footnote}
\makeatletter
\def\@fnsymbol#1{\ensuremath{\ifcase#1\or *\or \dagger\or \ddagger\or \mathsection\or \|\or **\or \dagger\dagger \or \ddagger\ddagger \else\@ctrerr\fi}}
\makeatother

\usepackage{xcolor}
\definecolor{cblue}{RGB}{0,70,140}
\definecolor{cgreen}{RGB}{100,140,0}
\definecolor{cred}{RGB}{190,10,50}

\usepackage[shortlabels]{enumitem}
\setlist{itemsep=0ex,topsep=0ex,parsep=0.4ex}

\usepackage{amsmath,amsfonts,amssymb,amsthm,mathtools,bbm,centernot}
\usepackage[hyphens]{url} % Line breaking of urls
\usepackage[hidelinks,colorlinks,pagebackref]{hyperref} % Coloring of hyperlinks and backreferences in bibliography, must be loaded after url
\hypersetup{citecolor=cgreen,linkcolor=cblue,urlcolor=cblue}
\usepackage[capitalise,nameinlink,noabbrev,compress]{cleveref} % References in the document; must be loaded after hyperref and amsmath

\renewcommand*{\backref}[1]{}
\renewcommand*{\backrefalt}[4]{
	\ifcase #1 Not cited.%
	\or $\uparrow$#2%
	\else $\uparrow$#2%
	\fi%
}

\let\oldbibliography\bibliography
\renewcommand{\bibliography}[1]{
  {
    \hypersetup{linkcolor=cred}
    \bibliographystyle{bibstyle}
    \oldbibliography{#1}
  }
}

\theoremstyle{plain}
\newtheorem{theorem}{Theorem}[section]
\newtheorem{lemma}[theorem]{Lemma}

\newtheorem{question}[theorem]{Question}

\theoremstyle{definition}
\newtheorem{definition}[theorem]{Definition}

\makeatletter
\renewenvironment{proof}[1][\proofname]
{\par\pushQED{\qed}
	\normalfont\topsep6\p@\@plus6\p@\relax\trivlist
	\item[\hskip\labelsep\bfseries#1\@addpunct{.}]
	\ignorespaces}
{\popQED\endtrivlist\@endpefalse}
\makeatother

\let\emptyset\varnothing

\newcommand{\cO}{\mathcal{O}}

\newcommand{\dist}{\operatorname{dist}}
\newcommand{\N}{N}

\title{Domination--packing ratio for planar and unit disk graphs}

\author{ W. Cames van Batenburg\affiliation{D\'epartement d'Informatique, Universit\'e libre de Bruxelles, Belgium (\textsf{\href{mailto:w.p.s.camesvanbatenburg@gmail.com}{w.p.s.camesvanbatenburg@gmail.com}}). Supported by the Belgian National Fund for Scientific Research (FNRS).}}
\date{}

\begin{document}
    \maketitle

\begin{abstract}
%We prove that for every planar graph, the domination number is at most $5$ times the packing number. This improves the bound of $7$ due to D\'ucz and Gujgiczer, who in turn improved the bounds $10$ due to Bonamy, Csik\'os, Gujgiczer and Yuditsky and the earlier $20$ due to B\"ohme and Mohar. We also show that for every unit disk graph, the domination number is at most $\frac{18\sqrt3}{\pi}\approx 9.924$ times the packing number. This improves the bound of $16$ due to Guti\'errez and Paul, who in turn improved the bound $32$ due to Bonamy, Csik\'os, Gujgiczer and Yuditsky. For both graph classes, the best known lower bound on the optimal constant remains $3$.

%We prove that for every planar graph, the domination number is at most $5$ times the packing number. We also show that for every unit disk graph, the domination number is at most $\frac{18\sqrt3}{\pi}\approx 9.924$ times the packing number. This improves upon bounds of D\'ucz and Gujgiczer, and Guti\'errez and Paul, who in turn lowered bounds of Bonamy, Csik\'os, Gujgiczer and Yuditsky, and B\"ohme and Mohar. For both graph classes, the best known lower bound on the optimal constant remains $3$.\\

The \emph{domination number} $\gamma(G)$ of a graph $G$ is the smallest possible size of a vertex set that intersects every radius-$1$ ball of $G$, and the \emph{packing number} $\rho(G)$ is the maximum number of pairwise vertex-disjoint radius-$1$ balls.
We prove that $\frac{\gamma(G)}{\rho(G)}\le 5$ for every planar graph, and $\frac{\gamma(G)}{\rho(G)} \le \frac{18\sqrt3}{\pi}\approx 9.924$ for every unit disk graph, thus yielding Erd\H{o}s--P\'osa-type bounds for the hypergraph of radius-$1$ balls in the two graph classes. This improves upon results of Guti\'errez and Paul, and D\'ucz and Gujgiczer, who in turn lowered bounds of Bonamy, Csik\'os, Gujgiczer and Yuditsky, and B\"ohme and Mohar. For both graph classes, the best known lower bound on the optimal constant remains $3$.\\

\end{abstract}

\section{Introduction}
Duality between packing and covering is a well-studied theme in structural graph theory: given a family of `nice' objects, one seeks either many pairwise disjoint nice objects, or a small set that intersects every nice object.  Statements of this type fall under the umbrella of \emph{Erd\H{o}s--P\'osa} theorems, which can often be encoded in terms of some class $\mathcal{H}$ of hypergraphs.
For every hypergraph $H\in \mathcal{H}$,  the nice objects correspond to the hyperedges of $H$. If $\rho$ denotes the maximum number of pairwise disjoint hyperedges and $\gamma$ the minimum size of a vertex set that intersects every hyperedge of $H$, then a \emph{bounding function} for $\mathcal{H}$ (if it exists!) is a function $f$ such that $\gamma\le f(\rho)$ for every $H\in \mathcal{H}$.  In the foundational result of Erd\H{o}s and P\'osa~\cite{ErdosPosa65}, the hyperedges are the cycles of some fixed graph: they proved that every graph either contains $k$ vertex-disjoint cycles, or a set of $\cO(k\log k)$ vertices meeting every cycle. Thus, for cycles there exists a bounding function $f(\rho)=O(\rho  \log \rho)$. Since then, the same question has been studied for many more structured objects, such as directed cycles~\cite{ReedRobertsonSeymourThomas96}, matroid circuits~\cite{GeelenKabell09}, cycles in group-labelled graphs~\cite{GollinHendreyKwonOumYoo25} and graph minor models~\cite{RobertsonSeymour86, CvBHJR19}, see also the survey~\cite{RaymondThilikos17}.

A particularly relevant special case for this note is that of \emph{ball hypergraphs}.  Given a graph $G$, a ball hypergraph of $G$ has vertex set $V(G)$ and hyperedges given by graph balls $B_r(v)=\{u\in V(G):\operatorname{dist}_G(u,v)\le r\}$, possibly with different radii.  In~\cite{BousquetEtAl21}, Bousquet et al. proved that for every graph $\Gamma$ there exists a constant $c$ such that every ball hypergraph of a $\Gamma$-minor-free graph satisfies $\gamma\le c \cdot \rho$. In particular this gave a linear bounding function for ball hypergraphs of planar graphs, even
when not all balls have the same radii, confirming  a conjecture of Chepoi, Estellon and Vax\`es~\cite{ChepoiEstellonVaxes07}.  Earlier progress, with constants depending on the radius or on weak colouring numbers, was obtained by Dvo\v{r}\'ak~\cite{Dvorak13,Dvorak19}.

In this note we focus on the very special case that all balls have radius $1$. Then the covering problem reduces to the usual domination problem.  More precisely, from now on, for a graph $G$, the \emph{domination number} $\gamma(G)$ is the minimum size of a set $D\subseteq V(G)$ such that its closed neighbourhood $\N_G[D]$ is the whole vertex set $V(G)$, equivalently, such that $D$ has non-empty intersection with every radius-$1$ ball. And the \emph{packing number} $\rho(G)$ is the maximum size of a set $P\subseteq V(G)$ whose closed neighborhoods are pairwise disjoint, equivalently, whose vertices are pairwise at distance larger than $2$. Because the term `packing number' is heavily overloaded in graph theory, referring to numerous very different notions of disjoint structures, a common alternative name for $\rho(G)$ is the \emph{$2$-independence number}. 

Clearly $\rho(G)\le \gamma(G)$.  This inequality is sharp for trees~\cite{MeirMoon75} and strongly chordal graphs~\cite{Farber84},
but it can fail arbitrarily badly even in sparse graphs: Dvo\v{r}\'ak~\cite{Dvorak13}
constructed bipartite $3$-degenerate graphs with $2$-independence number $2$
and unbounded domination number. An elegant basic positive bound is
$\gamma(G)\le \Delta(G)\rho(G)$~\cite{HenningLowensteinRautenbach11}, where $\Delta(G)$ denotes the maximum degree of $G$. Studies of the case $\Delta(G)=3$ and other variants can be found in ~\cite{GomezGutierrez25, ChoKim24, BoyerWayneHenning26}. 
For line graphs there is a matching interpretation: $\gamma(L(H)$ is the minimum size of a maximal matching of $H$, whereas $\rho(L(H))$ is the maximum size of an induced matching; see~\cite{CamesVanBatenburg22} for related conjectures and context in the $\Delta(H)=3$ setting.
Recently, Bonamy, Dvo\v{r}\'ak, Michel and Mik\v{s}an\'ik~\cite{BonamyDvorakMichelMiksanik26} gave an exact characterization of the monotone graph classes of bounded average degree that admit a linear bounding function $\gamma(G) \le c \cdot \rho(G)$, for some constant $c$ that only depends on the class.\\

In a recent work, Bonamy, Csik\'os, Gujgiczer and Yuditsky~\cite{BonamyCsikosGujgiczerYuditsky25} undertook a systematic study of graph classes with constant domination-packing ratio.  Among other results, they proved that $\gamma(G)\le 10\,\rho(G)$ for planar graphs, thus lowering a $\gamma(G)\le 20 \,\rho(G)-9$ bound of B\"ohme and Mohar~\cite{BohmeMohar03}. They also showed that $\gamma(G)\le 32\,\rho(G)$ for every \emph{unit disk graph} $G$, which by definition is a graph admitting a representation $p:V(G)\to\mathbb R^2$ such that, for distinct vertices $u,v$, we have $uv\in E(G)$ if and only if the Euclidean distance between $p(u)$ and $p(v)$ is $\le 1$.

In two separate works, the bounds for both classes have since been improved.
D\'ucz and Gujgiczer~\cite{DuczGujgiczer26} improved the planar bound from $10$ to $7$, while the best known lower bound for planar graphs is $3$ due to a graph on nine vertices exhibited by MacGillivray and Seyffarth~\cite{MacGillivraySeyffarth96}. For unit disk graphs, Guti\'errez and Paul~\cite{GutierrezPaul26} improved the bound from $32$ to $16$ and also gave an infinite family with ratio $3$.  
In this note we further lower the upper bounds for both graph classes.

\begin{theorem}\label{thm:ordinary}
For every finite simple planar graph $G$,
$$
    \gamma(G)\leq 5\,\rho(G).
$$
\end{theorem}

Similar to the arguments in~\cite{DuczGujgiczer26, BonamyCsikosGujgiczerYuditsky25}, the proof proceeds via a stronger theorem in terms of refinements $\gamma_X(G)$ and $ \rho_X(G)$ of $\gamma(G)$ and $\rho(G)$. These keep track of domination and packing on an arbitrary subset $X\subseteq  V(G)$, which helps overcome the technical difficulty that a packing on a subgraph may no longer be a packing in the full graph because the addition of a vertex (plus its incident edges) may decrease distances. We proceed by minimum counterexample and use a discharging argument. The main reason why we were able to improve on the argument of~\cite{DuczGujgiczer26} is the simple yet effective Lemma~\ref{lem:neighbourslikeeachother}.

\begin{theorem}\label{thm:main}
For every finite unit disk graph $G$,
$$\gamma(G)\le \frac{18\sqrt3}{\pi}\,\rho(G) <10\,\rho(G).$$
\end{theorem}

The proof takes a maximal  independent (and hence dominating) set $I$ of the unit disk graph, considers the points $X\subseteq \mathbb{R}^2$  that represent $I$, and then uses a randomly shifted triangular grid to extract a large subset $Y\subseteq X$ of points that are at pairwise Euclidean distance $>2$. The vertices corresponding to $Y$ have disjoint closed neighborhoods and thus form a packing.

%On the algorithmic side, we remark that da Fonseca, de Figueiredo, Pereira de S\'a and Machado~\cite{DaFonsecaEtAl14} gave an efficient $\frac{43}{9}$-approximation algorithm for $\frac{\gamma(G)}{\rho(G)}$ for unit disk graphs $G$. This note does not improve upon that.
%Hmmm... it seems there also exist (1+\epsilon) PTAS, so I'm not sure of the state of the art. Let's leave these remarks in the comments just in case they are useful later.
%Efficient sub-$5$ approximation algorithms for the minimum dominating set problem in unit disk graphs: a $44/9$-approximation in $O(n+m)$ time from the adjacency representation, the same factor in $O(n\log n)$ time from a geometric representation, and a $43/9$-approximation in $O(n^2m)$ time from the adjacency representation.

\section{Planar graphs}
In this section we prove Theorem~\ref{thm:ordinary}.
The main difficulty in an inductive domination--packing argument is that deleting a vertex may change distances between the remaining vertices.  We use the following refinements to keep track of the vertices that are relevant on both sides of the inequality.

\begin{definition}\label{def:X}
Let $G$ be a graph and let $X\subseteq V(G)$.
\begin{enumerate}[label=\textup{(\roman*)}]
    \item A set $D\subseteq V(G)$ is an \emph{$X$-dominating set} if
    $X\subseteq \N_G[D].$ 
    The minimum cardinality of an $X$-dominating set is denoted by $\gamma_X(G)$.

    \item A set $P\subseteq X$ is an \emph{$X$-packing} if
    $\dist_G(p,q)\geq 3$
    for all distinct $p,q\in P$.  Equivalently, the closed neighborhoods of distinct vertices of $P$ are disjoint.  The maximum cardinality of an $X$-packing is denoted by $\rho_X(G)$.
\end{enumerate}
\end{definition}

Writing $\overline{X}=V(G)\setminus X$
for the complement of a set $X$ in $V(G)$, we note that Definition~\ref{def:X} is effectively the same as definitions of D\'ucz and Gujgiczer~\cite{DuczGujgiczer26}, who adapted a definition introduced by Bonamy, Csik\'os, Gujgiczer and Yuditsky~\cite{BonamyCsikosGujgiczerYuditsky25}. However, what we call an $X$-dominating set would be an $\overline{X}$-dominating set in their terminology. We have made this notational swap, because we found it more intuitive to work with an $X$-dominating set that is in fact a set that dominates $X$, and that an $X$-packing is a packing whose vertices must be chosen from $X$.\\ 

With the choice $X=V(G)$, Definition~\ref{def:X} yields
$\gamma_{V(G)}(G)=\gamma(G)$ and $\rho_{V(G)}(G)=\rho(G).$ Consequently, Theorem~\ref{thm:ordinary} will be a direct corollary of the following stronger statement.

\begin{theorem}\label{thm:refined}
For every finite simple planar graph $G$ and every $X\subseteq V(G)$,
$$\gamma_X(G)\leq 5\rho_X(G).$$
\end{theorem}

\subsection{A minimum counterexample}

Suppose for a contradiction that Theorem~\ref{thm:refined} is false, and choose a counterexample $(G,X)$ minimizing $|X|$, and subject to that minimizing $|V(G)|+|E(G)|$.   Here $X\subseteq V(G)$ is the set that must be dominated and from which the packing vertices must be chosen. If $G$ is disconnected, $G_1,\ldots,G_t$ are the components of $G$ and $X_i=X\cap V(G_i)$, then
$\gamma_X(G)=\sum_{i=1}^t\gamma_{X_i}(G_i) \leq 5\sum_{i=1}^t\rho_{X_i}(G_i) =5\rho_X(G)$, so we may assume that $G$ is connected.
 It is easy to see that a graph with at most one vertex is not a counterexample, so $|V(G)|\geq 2$. The reductions in Lemma~\ref{lem:reductions} are also deduced in~\cite{DuczGujgiczer26}, but we include the proof for completeness.\\

\begin{lemma}\label{lem:reductions}
The following statements hold.
\begin{enumerate}[label=\textup{(\alph*)}]
    \item $\overline{X}$ is independent.
    \item Every vertex of $X$ has degree at least $6$.
    \item Every vertex of $\overline{X}$ has degree at least $2$.
\end{enumerate}
\end{lemma}

\begin{proof}
For (a), suppose that $y_1y_2\in E(G)$ with $y_1,y_2\in \overline{X}$, and let $H=G-y_1y_2$.  By minimality, there are an $X$-dominating set $D'$ and an $X$-packing $P'$ in $H$ such that
$|D'|\leq 5|P'|$.
Restoring the edge $y_1y_2$ (or adding any edge, for that matter) cannot spoil $X$-domination. Thus  $D'$ is also an $X$-dominating set in $G$.
Restoring $y_1y_2$ cannot spoil the packing property either. Indeed, a hypothetical new path of length $\le 2$ must traverse the new edge $y_1y_2$ and hence has at least one endpoint in $\{y_1,y_2\}$. Since $P'\subseteq X$ and $y_1,y_2\notin X$, no such path has both endpoints in $P'$. Thus  $P'$ is also an $X$-packing in $G$, a contradiction.

For (b), let $v\in X$ and suppose that $d_G(v)\leq 5$.  Since $G$ is connected and has at least two vertices, $d_G(v)\geq 1$.  Set
$$X'=\{u\in X\setminus\{v\}:\dist_G(u,v)\geq 3\}.$$
By minimality of $|X|$, there are an $X'$-dominating set $D'$ and an $X'$-packing $P'$ in $G$ satisfying
$|D'|\leq 5|P'|$.
Define
$$D=D'\cup \N_G(v) \qquad\text{and}\qquad P=P'\cup\{v\}.$$

We first verify domination.  Every vertex of $X'$ is dominated by $D'$ in $G$.  Every vertex of $X\setminus X'$ is either $v$ itself or is at distance at most two from $v$ in $G$, and is therefore dominated by $\N_G(v)$.  Here $v$ is dominated because $d_G(v)\geq1$.  Hence $D$ is an $X$-dominating set of $G$.

Next, $P\subseteq X$.  Since $P'\subseteq X'$, every vertex of $P'$ is at distance at least three from $v$ in $G$.  Since $P'$ is an $X'$-packing in $G$, it follows that $P$ is an $X$-packing of $G$.  Finally,
$|D|\leq |D'|+d_G(v)\leq 5|P'|+5=5|P|$, a contradiction.

For (c), let $y\in \overline{X}$ have degree at most one, and put $H=G-y$.  Since $y\notin X$, the set $X$ is contained in $V(H)$.  By minimality, there are an $X$-dominating set $D'$ and an $X$-packing $P'$ in $H$ such that $|D'|\leq5|P'|$.  The set $D'$ remains an $X$-dominating set after $y$ is restored.  Restoring an isolated vertex or a leaf cannot shorten the distance between two vertices already present in $H$, so $P'$ remains an $X$-packing.  This contradicts the choice of $(G,X)$.
\end{proof}

The next lemma is the key local reducibility statement. In particular it implies that after removal of a vertex $y\in \overline{X}$,  at least two of its neighbours must be at distance $\ge 3$.

\begin{lemma}\label{lem:neighbourslikeeachother}
For every $y\in \overline{X}$ and every $a\in\N_G(y)$, there exists
$c\in\N_G(y)\setminus\{a\}$ such that
$$\dist_{G-y}(a,c)\geq3.$$
\end{lemma}

\begin{proof}
By Lemma~\ref{lem:reductions}(c), $y$ has at least two neighbors.  Suppose for a contradiction that some $a\in\N_G(y)$ satisfies
$\dist_{G-y}(a,c)\leq 2$ for every $c\in\N_G(y)\setminus\{a\}.$
Let $H=G-ya$.  By minimality, there are an $X$-dominating set $D'$ and an $X$-packing $P'$ in $H$ with $|D'|\leq5|P'|$.  As $X$-domination is preserved under edge-addition, $D'$ is also an $X$-dominating set in $G$.
Suppose that restoring edge $ya$ spoils the packing property.  Since $P'\subseteq X$ and $y\in \overline{X}$, we have $y\notin P'$.  Therefore the only possible new path of length at most two with both endpoints in $P'$ is a path $a y c$, where $c\in\N_G(y)\setminus\{a\}$ and $a,c\in P'$.  By assumption, however, $a$ and $c$ are already at distance at most two in $G-y$, which is a subgraph of $H$.  This contradicts that $P'$ is a packing in $H$.  Hence $P'$ is also an $X$-packing in $G$, a contradiction.
\end{proof}

\subsection{Discharging}

Fix a plane embedding of $G$.  The length $\ell(f)$ of a face $f$ is the length of its facial boundary walk, with repeated vertices and edges counted with multiplicity.

\begin{definition}[Corner and $\overline{X}$-corner]\label{def:corner}
Write the cyclic boundary walk of a face $f$ as
$$(v_0,e_1,v_1,\ldots,e_{\ell(f)},v_{\ell(f)}=v_0).$$
For $0\leq i<\ell(f)$, the vertex $v_i$ in this cyclic walk together with the angular sector of $f$ between the preceding and following edges, is a \emph{corner of $f$ at $v_i$}.  It is an \emph{$\overline{X}$-corner} if $v_i\in \overline{X}$. 
%Moreover, $v_i$ is called the vertex determining the $\overline{X}$-corner. 
Note that a cutvertex may determine several corners of the same face.  Let $k(f)$ be the number of $\overline{X}$-corners of $f$, counted with multiplicity. 
\end{definition}

While in general cut vertices may prevent a face from being bounded by a cycle, we do not have that problem for faces of length three or four.  Recalling that $G$ is simple loopless by assumption and has minimum degree at least two by Lemma~\ref{lem:reductions}, the next lemma is immediate.
\begin{lemma}\label{lem:faces}
Every face has length $\geq 3$, and every face of length $\in \{3,4\}$ is bounded by a cycle.
\end{lemma}

%\begin{proof}
%Lemma~\ref{lem:reductions} gives minimum degree at least two.  A facial walk can immediately reverse an edge only at a vertex of degree one, so no facial walk of $G$ has an immediate reversal.  In a simple graph, a closed walk of length one or two, or a noncyclic closed walk of length three or four, necessarily contains a loop, a pair of parallel edges, or an immediate reversal.
%\end{proof}

Since $\overline{X}$ is independent, no two consecutive corners of a facial boundary walk are $\overline{X}$-corners.  Therefore
\begin{equation}\label{eq:k}
k(f)\leq\left\lfloor\frac{\ell(f)}2\right\rfloor.
\end{equation}

Now we start the discharging argument. We give every vertex $v$ and face $f$ initial charge
$$\mu(v)=d_G(v)-6  \quad \quad\text{ and } \quad \quad
    \mu(f)=2\ell(f)-6.$$
Euler's formula, together with $\sum_v d_G(v)=\sum_f\ell(f)=2|E(G)|$ then gives that the total sum of the charges
$\sum_{v\in V(G)}\mu(v)+\sum_f\mu(f)
    =(2|E|-6|V|)+(4|E|-6|F|)=-12$ is negative.
As is the standard approach, we will redistribute the charges so that in the end each vertex and face will have nonnegative charge,  while also always keeping the total sum of the charges the same $-12$, thus arriving at a contradiction. For that redistribution, we will need only one rule.\\

\textbf{Discharging rule}. 
%\emph{ Each face $f$ with $\ell(f)\geq4$ and $k(f)>0$ distributes all of its charge $2\ell(f)-6$ equally among its $\overline{X}$-corners. 
%Each face $f$ with $\ell(f)\geq4$ and $k(f)>0$ sends charge $\mu(f)/k(f)$ to each of its $\overline{X}$-corners. Subsequently, each vertex in $\overline{X}$ receives the charges of all its $\overline{X}$-corners.
\emph{Every face $f$ with $\ell(f)\geq4$ and $k(f)>0$ distributes its entire charge equally among its $\overline{X}$-corners: each such corner receives $\mu(f)/k(f)$. The charge received by a corner is subsequently passed to the vertex determining that corner.}
\\

Note that a face with $k(f)=0$ keeps its charge, and triangular faces have charge zero. Every $\overline{X}$-corner of a face $f$ receives charge $(2\ell(f)-6)/k(f)$, which by~\eqref{eq:k} is at least
\begin{equation}\label{eq:receipt}
    1\quad\text{if }\ell(f)=4,
    \qquad
    2\quad\text{if }\ell(f)\geq5.
\end{equation}
Indeed, a face of length four has charge $2$ and at most two $\overline{X}$-corners, while for $\ell(f)\geq5$ we have $2\ell(f)-6\geq2\lfloor\ell(f)/2\rfloor\geq2k(f)$.  Every face therefore finishes with nonnegative charge.  Every vertex of $X$ also finishes with nonnegative charge by Lemma~\ref{lem:reductions}(b), as does every vertex of $\overline{X}$ of degree at least six.

Let $y\in \overline{X}$ have degree $d\in\{2,3,4,5\}$.  List its neighbors cyclically as $v_1,\ldots,v_d$, and let $f_i$ be the face in the corner between $yv_i$ and $yv_{i+1}$, with indices modulo $d$.  Define
$$w_i=
    \begin{cases}
        0,&\text{if }\ell(f_i)=3,\\
        1,&\text{if }\ell(f_i)=4,\\
        2,&\text{if }\ell(f_i)\geq5.
    \end{cases}
    $$
By~\eqref{eq:receipt}, the charge received by $y$ is at least $\sum_{i=1}^{d} w_i$.  If $w_i=0$, then Lemma~\ref{lem:faces} implies $v_iv_{i+1}\in E(G-y)$, so then we say that the corner between $yv_i$ and $yv_{i+1}$ is \emph{triangular}.  If $w_i=1$, the same lemma implies there is a path of length two between $v_i$ and $v_{i+1}$ in $G-y$, so then we say the corner between $yv_i$ and $yv_{i+1}$ is \emph{quadrilateral}.

The next and last lemma produces a lower bound on the total charge that a degree $\le 5$ vertex of $\overline{X}$ receives.

\begin{lemma}\label{lem:weights}
For every $y\in \overline{X}$ of degree $d\leq5$,
$$\sum_{i=1}^d w_i\geq6-d.$$
\end{lemma}

\begin{proof}
Suppose for a contradiction that $\sum_{i=1}^d w_i\leq5-d$. Then for every case below, we derive that all pairs of neighbors of $y$ are at distance at most two in $G-y$, contradicting Lemma~\ref{lem:neighbourslikeeachother}.

If $d=5$, then every $w_i$ is zero, so every corner is triangular.  Thus the neighbors of $y$ contain the cycle $v_1v_2v_3v_4v_5v_1$ in $G-y$, and every two of them are at distance at most two.

If $d=4$, then all corners are triangular except possibly one quadrilateral corner.  Relabeling if necessary, we may assume that the corner between $yv_4$ and $yv_1$ is the only corner that might be quadrilateral. The vertices $v_1, v_4$ are at distance at most two, and $v_1v_2v_3v_4$ must be a path.  Hence every two neighbors of $y$ are at distance at most two in $G-y$.

If $d=3$, then either every $w_i\leq1$, in which case each of the three neighbor-pairs is at distance at most two, or one weight is two and the other two are zero.  In the latter case the two triangular edges form a path of length two between the remaining pair.  Again every two neighbors are at distance at most two in $G-y$.

If $d=2$, then at least one of the two weights must be $\le 1$.
The corner corresponding to that weight is triangular or quadrilateral, so there is a path of length at most two between the two neighbors of $y$ in $G-y$.
%The face corresponding to that weight is incident to $y$ and has length at most four, so there is a path of length at most two between the two neighbors of $y$ in $G-y$. 

\end{proof}

\begin{proof}[Proof of Theorem~\ref{thm:refined}]
A vertex $y\in \overline{X}$ of degree $d\leq5$ starts with charge $d-6$ and, by Lemma~\ref{lem:weights}, receives at least $6-d$.  Hence it finishes with nonnegative charge.  As observed above, every other vertex and every face also finishes with nonnegative charge.  This contradicts the total charge $-12$.  Therefore a minimum counterexample does not exist, and $\gamma_X(G)\leq5\rho_X(G)$
holds for every finite simple planar graph $G$ and every $X\subseteq V(G)$.
\end{proof}

\begin{proof}[Proof of Theorem~\ref{thm:ordinary}]
Apply Theorem~\ref{thm:refined} with $X=V(G)$.
\end{proof}

\section{Unit disk graphs}
%\section{Finding a large $2$-separated set inside an independent set}
In this section we prove Theorem~\ref{thm:main}.
A subset $P\subset \mathbb{R}^2$ is called \emph{$c$-separated} if $d(x,y) > c$ for all distinct $x,y \in P$. Here $d(x,y)$  denotes the Euclidean distance between $x$ and $y$. To prove Theorem~\ref{thm:main}, we first derive an auxiliary lemma that finds a large $2$-separated set inside a given $1$-separated set.
The averaging step in the argument is standard and can be viewed as an instance of Blichfeldt’s principle in the geometry of numbers~\cite{Cassels97}. Related random shift constructions for unit disk graphs appear in~\cite{CaragiannisEtAl07} and~\cite{EsperetJulliotDeMesmay23}. 
%To keep the proof short, we did not optimize every step.

\begin{lemma}\label{lem:randomdisks}
Let $X\subset \mathbb R^2$ be a finite $1$-separated point set.
Then there exists a $2$-separated subset $Y\subseteq X$ of size $|Y|\ge \frac{\pi}{18\sqrt 3}\,|X|$.
\end{lemma}

\begin{proof}
Let $\Lambda$ be the triangular lattice with shortest nonzero vector of length
$3$. For instance, take
$\Lambda = \left\{
        m(3,0)+n\left(\frac32,\frac{3\sqrt3}{2}\right)
        :m,n\in\mathbb Z
        \right\}.$
A fundamental parallelogram of $\Lambda$ has area $\det(\Lambda)=\frac{9\sqrt3}{2}$.
Let
$$ S=\bigcup_{\lambda\in\Lambda} B(\lambda,1/2),$$
where $B(\lambda,1/2)$ denotes the open Euclidean disk of radius $1/2$ centred
at $\lambda$. For a translation vector $t\in\mathbb R^2$, define
$$Y_t=X\cap (S+t).$$
We first observe that $Y_t$ is always $2$-separated. Indeed, suppose
$x,x'\in Y_t$ are distinct. If $x$ and $x'$ lie in the same open disk of radius
$1/2$, then $d(x,x')<1$,
contradicting the hypothesis on $X$. Hence they lie in two distinct translated
disks. The centres of distinct disks in $S+t$ are at distance at least $3$.
Since the disks are open and have radius $1/2$, we get $d(x,x')>3-\frac12-\frac12=2$.
Thus $Y_t$ is $2$-separated.

%This is a standard argument, but here are more details if someone asks:
%Now choose $t$ uniformly at random from a fundamental parallelogram $F$ of $\Lambda$. Fix $x\in X$. Since $t$ is uniform modulo $\Lambda$ and $x\in S+t$, the point $x-t \in S$ is also uniform modulo $\Lambda$. Modulo $\Lambda$, the periodic set $S$ consists of one disk of radius $1/2$ inside a fundamental parallelogram, possibly cut into pieces by the boundary of $F$. These pieces have total area $\operatorname{area}(B(0,1/2))=\pi/4$. The projection is injective because the shortest nonzero vector of $\Lambda$ has length $3$. Hence $$\mathbb P(x\in S+t) = \frac{\operatorname{area}(B(0,1/2))}{\det(\Lambda)} =\frac{\pi}{18\sqrt3}.$$

Now choose $t$ uniformly at random from a fundamental parallelogram of
$\Lambda$. For every fixed point $x\in X$, the probability that
$x\in S+t$ is exactly the density of $S$, namely
$$\frac{\operatorname{area}(B(0,1/2))}{\det(\Lambda)} = \frac{\pi/4}{9\sqrt3/2} = \frac{\pi}{18\sqrt3}.$$

Therefore the expected number of points in $Y_t$ is $\mathbb E|Y_t|= \sum_{x\in X}\mathbb P(x\in S+t) = \frac{\pi}{18\sqrt3}\,|X|$.
So for at least one translation $t$, we have $|Y_t|\ge \frac{\pi}{18\sqrt3}\,|X|$.
\end{proof}

\begin{proof}[Proof of Theorem~\ref{thm:main}]
Fix a unit disk representation  $p:V(G)\to\mathbb R^2$
such that, for distinct vertices $u,v\in V(G)$, we have
$uv\in E(G)$ if and only if $d(p(u),p(v))\le 1.$
Let $I$ be a maximal independent set of $G$. Due to maximality $I$ is dominating, so $\gamma(G)\le |I|$.
Since $I$ is independent in $G$, we have $d(p(u),p(v))>1$ for all distinct $u,v\in I$.
Applying Lemma~\ref{lem:randomdisks} to the point set $p(I)$, we obtain a $2$-separated subset $Y\subseteq p(I)$ with $|Y|\ge \frac{\pi}{18\sqrt3}\,|I|$.

%Originally I wrote ``Let $P=p^{-1}(Y)\subseteq I$ be the corresponding set of vertices of $G$." Pedantically, this is not quite right because $p$ might not be injective. Multiple disks could have the same center. So need P=\{u\in I:p(u)\in Y\} instead.
Let $ P=\{u\in I:p(u)\in Y\}\subseteq I$ be the corresponding set of vertices of $G$. 
We claim that $P$ is
a packing in $G$. Indeed, if distinct vertices $u,v\in P$ had a common vertex
in their closed neighbourhoods, say
 $w\in N_G[u]\cap N_G[v]$,
then due to the triangle inequality $ d(p(u),p(v))
        \le
        d(p(u),p(w))+d(p(w),p(v))
        \le 2,$
contradicting that $Y$ is $2$-separated. Thus $P$ is a packing, and so
$       \rho(G)\ge |P| \ge
        \frac{\pi}{18\sqrt3}\,|I|  \ge \frac{\pi}{18\sqrt3}\,\gamma(G).$
\end{proof}

\section{Concluding remarks}

\subsection{Planar graphs}
MacGillivray and Seyffarth~\cite{MacGillivraySeyffarth96} gave an example of a planar graph $H$ with diameter $2$ and domination number $3$. Later, Goddard and Henning~\cite{GoddardHenning02} showed that $H$ is the unique such graph. 
In part inspired by the fact that $\rho(H)=1$ and $\gamma(H)=3$, D\'ucz and Gujgiczer~\cite{DuczGujgiczer26} conjectured that $\gamma(G)\le 3 \cdot \rho(G)$ is the optimal bound for planar $G$. 
%The upper bound holds true for every maximal outerplanar graph~\cite{GomezGutierrez25}. This was improved to a factor $2$ by~\cite{BonamyCsikosGujgiczerYuditsky25}
We were not yet able to reduce the multiplicative constant below $5$, let alone make it $3$, but it does appear there is a way to strengthen our Theorem~\ref{thm:ordinary} to a bound for which the same $H$ is sharp, namely $\gamma(G)\leq 5\cdot \rho(G) -2$ for every planar $G$.  Most of the proof would stay the same, but the natural modification $\gamma_X(G)\leq 5\cdot \rho_X(G)-2$ of Theorem~\ref{thm:refined}  needs to be stated for all $G,X$ that are \emph{non-empty}. This necessitates a separate treatment of the case that $X'=\emptyset$ and $\rho_X(G)=1$ in the proof of Lemma~\ref{lem:reductions}(b). This in turn needs a proof that if the packing number $\rho_X(G)$ equals $1$, then $\gamma_X(G) \le 3$. 
For $X=V(G)$, this follows from a theorem of MacGillivray and
Seyffarth~\cite[Theorem~1]{MacGillivraySeyffarth96}. For its extension to arbitrary $X$, the separation and connector arguments of~\cite{GoddardHenning02}, particularly Lemma 13, appear relevant but in need of multiple pages of proof. The latter approach was suggested by ChatGPT 5.6 Sol in the final phase of this project, when we uploaded this draft and asked it to find this specific improvement. Since we prefer to keep the note short and easy to follow, we do not pursue that technical extension here. 
Upon asking the model to lower the multiplicative constant $5$, it made some structural progress towards a potential ratio $9/2$, but was still far from closing any argument.

In the introduction we discussed that for every $t$, there exists $c_t>0$ such that $\gamma(G) \leq c_t \cdot \rho(G)$ for every $K_t$-minor-free graph $G$, due to Bousquet et al.~\cite{BousquetEtAl21}. Their argument was designed to work for arbitrary ball hypergraphs and consequently the constant $c_t$ obtained there is quite large, of the order $\exp(t(\log t)^{3/2})$. As planar graphs are $K_5$-minor-free, the following is a natural optimization problem.

\begin{question}
What is the smallest constant $c>0$ such that $\gamma(G)\leq c \cdot \rho(G)$ for every $K_5$-minor-free graph $G$?
\end{question}
%How does it compare to Theorem~\ref{thm:ordinary}?

The graph $H$ mentioned above shows $3\le c$ and a result of B\"ohme and Mohar~\cite{BohmeMohar03} for graphs excluding $K_{4,4}$ as a minor immediately implies $c\leq 31$. Could we have $c\le 5$?
%, as in Theorem~\ref{thm:ordinary}??

\subsection{Unit disk graphs}
Any improvement of the constant $c=\frac{\pi}{18 \sqrt{3}}$ in Lemma~\ref{lem:randomdisks} directly implies an improvement $\gamma(G) \le \frac{1}{c} \rho(G)$ in Theorem~\ref{thm:main}. We suspect that $c\approx\frac{1}{6}$ is achievable, but that the "hail shot" approach of choosing some fixed lattice independent of $P$, and then taking a (uniformly) random translation, has almost been optimized. Preliminary computations indicate that small gains are possible by slightly adapting the shape and size of the radius $1/2$ balls around the lattice points, or by taking a random rotation in addition to a random translation. However, these modifications make the proofs much longer, more technical and ugly, in return for only minor improvements. So far, the best ratio that we could heuristically obtain in this way is 
$\gamma(G)\le  \frac{9\pi}{9\sqrt3-4\pi}\,\rho(G) \approx 9.36 \,\rho(G)$ via randomly rotated and translated regular hexagons.
For significant further improvements, a more direct approach tailored to the specific structure of the given point sets $P$ and the Euclidean metric may be necessary.

\paragraph{Tool and computational resource disclosure.}
The author made use of DuckDuckGo, Google, Overleaf, Python, ChatGPT, Claude, laptop, coffee and human neurons.

\bibliography{bib}
\end{document}